\documentclass[a4paper,10pt,reqno]{amsart}
\usepackage{amssymb,amscd,a4,textcomp, latexsym,psfrag,amsmath, calc, multirow, color, cancel, graphicx, verbatim}
\usepackage[english]{babel}
\usepackage[utf8]{inputenc}
\usepackage{subfigure}
\usepackage[all, cmtip]{xy}
\usepackage{diagrams}
\usepackage[normalem]{ulem}

\reversemarginpar

\usepackage[usenames,dvipsnames]{pstricks}
\usepackage{epsfig}
\usepackage{pst-grad} 
\usepackage{pst-plot} 

\usepackage{hyperref}
\usepackage[active]{srcltx}

\newcommand{\de}{\mathrm{d}}
\newcommand{\rank}{\mathrm{rank}}

\newcommand{\N}{\mathbb{N}}
\newcommand{\Z}{\mathbb{Z}}
\renewcommand{\P}{\mathbb{P}}
\newcommand{\C}{\mathbb{C}}

\newcommand{\Pic}{\mathrm{Pic}}
\newtheorem{lemma}{Lemma}[section]
\newtheorem{teo}[lemma]{Theorem}
\newtheorem{cor}[lemma]{Corollary}
\newtheorem{prop}[lemma]{Proposition}
\theoremstyle{definition}

\title{On the surjectivity of weighted Gaussian maps}

\author{E. Ballico, L. Pernigotti}

\begin{document}

\maketitle

\vspace{0,4cm}
\begin{center}\parbox{0.8\textwidth}{
\begin{small}
ABSTRACT: We study the surjectivity of suitable weighted Gaussian maps $\gamma_{a,b}(X,L)$ which provide a natural generalization of the standard Gaussian maps and encode
the local geometry of the locus $\mathfrak{Th}^r_{g,h}\subset \mathcal{M}_g$ of curves endowed with an $h$-th root $L$ of the
canonical bundle satisfying $h^0(L)\geq r+1$. In particular, we get a bound on the dimension of its Zariski tangent space, which turns out to be sharp in the special case $r=0$.  
Finally, we describe this locus in the case of complete intersection curves.

\vspace{0,4cm}


\vspace{0,4cm}

\end{small}
}\end{center}

\section{Introduction}

The theory of Gaussian maps on curves was developed by Jonathan Wahl in \cite{W87} and \cite{W90}. We recall that they are defined as
\[
\Phi_L:R(L)\to H^0(\Omega_X^1\otimes L^{\otimes 2}), \qquad  \sigma\otimes \tau  \mapsto  \tau\de\sigma-\sigma\de\tau,
\]
where $X$ is a smooth projective variety over $\C$, $L\in \Pic(X)$ a line bundle of positive degree and $R(L)$ is the kernel of the multiplication
map $H^0(L)\otimes H^0(L)\to H^0(L^{\otimes 2})$. Furthermore, as for instance in \cite{Far05}, it makes sense to define the Gaussian map also as the restriction $\psi_L:={\Phi_L}_{|\Lambda^2 H^0(L)}$ since the first
map always vanishes on symmetric tensors. The interpretation of the name ``Gaussian'' can be found for instance in \cite{W90}: if 
$X=C$ is a curve embedded in a projective space $\P^n$ and $L=\mathcal{O}_C(1)$, one can consider the Gauss mapping $C\to \mathrm{Gr}(1,n)$ sending each point of the curve 
to its tangent line in $\P^n$. The composition of this map with the Pl\"ucker embedding $\mathrm{Gr}(1,n)\hookrightarrow \P^N$ gives rise to the ``associated curve'' $\phi:C\to\P^N$ and the
restriction of the hyperplane section $\phi^*$ of $\phi$ corresponds the so-called Gaussian map $\psi_L$.

The original interest in these maps came from the study of $\Phi_{K_C}$, where $K_C$ is the canonical bundle on a smooth curve (see for instance \cite{CHM88} and \cite{CLM00})
and in general they have been explored in particular when $X$ is a curve, in relation with the deformation theory of the vertex of the cone over $X$ (see \cite{W87}).
More recently, the first named author joint with Claudio Fontanari provided a generalization of these maps in \cite{BF06}. They defined the so-called ``weighted Gaussian maps'' as
\begin{equation}\label{eq:gammaab}\begin{array}{cccc}
\gamma_{a,b}(X,L): & H^0(X,L^{\otimes a})\otimes H^0(X,L^{\otimes b}) & \to &  H^0(\Omega^1_X\otimes L^{\otimes a}\otimes L^{\otimes b})\\
		   & \sigma\otimes \tau & \mapsto & b\tau\de\sigma-a\sigma\de\tau.
\end{array}\end{equation}
where $X$ is always a smooth projective variety, $L\in\mathrm{Pic}(X)$ and $a, b>0$ are two positive integers. When $a=b=1$ we recover the standard Gaussian map.

In the first section, following the approach of \cite{W90}, we investigate the surjecti\-vi\-ty of the map $\gamma_{a,b}(X,L)$ by studying first the weighted Gaussian maps
for $X=\P^n$ and $L=\mathcal{O}_{\P^n}(e)$, with $e\in \N$ a generic positive
integer, and then the restriction map given by $H^0(\P^n,\Omega^1_{\P^n}(t))\to H^0(X,\Omega^1_X(t))$. 
In particular, we prove a general result (Theorem \ref{teo}) from which we deduce the following.
 
\vspace{0.5cm}
\noindent\textbf{Corollary \ref{cor1}.}\textit{
 Let $C$ be a smooth curve of genus $g$ and $L\in\mathrm{Pic}(C)$ a line bundle of degree $\deg(L)\geq 2g+2$. Denote by $X\subset \P^n$ the linearly normal embedding of $C$ induced by $|L|$.
 Then the map $\gamma_{a,b}(X,\mathcal{O}_X(1))$ is surjective for all positive integers $a,b$ such that $a+b\geq 3$.}
 
\vspace{0.5cm}
\noindent\textbf{Corollary \ref{cor2}.}\textit{
 Let $C$ be a smooth curve of genus $g\geq 3$. Assume that $C$ is neither hyperelliptic nor trigonal nor isomorphic to a plane quintic and let $X\subset \P^{g-1}$ be its canonical model. Then
 the map $\gamma_{a,b}(X,\omega_X)$ is surjective for all $a,b>0$ such that $a+b\geq 4$.
}

\vspace{0.5cm}
\indent In the second section we relate the weighted Gaussian maps $\gamma_{1,h-1}(C,L)$, $h\ge 2$, to the locus $\mathfrak{Th}^r_{g,h}\subset\mathcal{M}_g$ defined,
 for $h\ge 2$ an arbitrary integer, by
 \begin{equation}\label{deflocus}
 \mathfrak{Th}^r_{g,h}:=\{[C]\in \mathcal{M}_g \ \mid \ \exists L\in\mathrm{Pic}(C) \text{ s.t. } h^0(C,L)\geq r+1, \ hL=K_C\}.
 \end{equation}
 It is a generalization for $h$-spin curves of the locus studied by J. Harris in \cite{H82} for theta-characteristics. 
 The case in which $h$ is even  was considered in \cite{Fon02}, but the proof of the odd case is identical (see Theorem \ref{teofon}). 
 In particular, J. Harris proved that each component of the locus
 \[
 \mathcal{S}^r_g:=\{[C,L]\in \mathcal{S}_g \ | \ h^0(L)\geq r+1, \quad h^0(L)\equiv r+1 \quad \mathrm{mod}\quad 2\}
 \]
 has codimension at most ${r+1 \choose 2}$ in $\mathcal{S}_g$. We recall that $\mathcal{S}_g$ is the moduli space of pairs $[C,L]$ where $C$ is a genus $g$ curve and $L\in\Pic(C)$ is a 
 theta-characteristic on $C$, that is, a square root of its canonical bundle. In \cite{Far05} G. Farkas showed that for $r=1,2,\ldots,9$ and $11$ there exists an explicit integer $g(r)$ such that
 for all $g\geq g(r)$ the moduli space $\mathcal{S}^r_g$ has at least one component of codimension exactly ${r+1 \choose 2}$ and he made a conjecture (recently
 proved by L. Benzo in \cite{Be13}) on the existence of a component attaining
 the maximum codimension for any $r\geq 1$ and $g\geq {r+2 \choose 2}$. His proof is based on the connection between Gaussian maps and spin curves provided by a tangent space computation done by Nagaraj in \cite{N90},
 whose main ingredient is
 the identification $T_{[C,L]}\mathcal{S}^r_g\cong\mathrm{Im}(\psi_L)^\perp$. In this work we use the analogous relation between the weighted Gaussian map and the locus
 $ \mathfrak{Th}^r_{g,h}$
 proved in \cite[Theorem 3]{Fon02} for $h$ even: there is an identification between the tangent space $T_C\mathfrak{Th}^r_{g,h}$ and the dual of the cokernel $\mathrm{Coker}(\gamma_{1,h-1}(C,L))$.
 As claimed before, the same relation holds identically also when $h$ is odd. Using these facts we prove the following result.
   
\par
\quad\\
\textbf{Theorem \ref{teo:sect3}.}\textit{
  For every $g,h\geq 2$ and every $[C] \in \mathfrak{Th}^0_{g,h}$ with an $h$-theta $L$ satisfying $h^0(L)=1$, the Zariski tangent space at $[C]$
  has codimension $(g-1)(h-2)/h$ in the tangent space $H^0(C,K _C^{\otimes 2})^\vee $ of the local deformation space of $C$.
}

\par
\quad\\
\indent We conclude our work with a focus on complete intersection curves. In Theorem \ref{teofinal} we
 prove that, for general $m$ and $r$, if $\mathfrak{Th}^r_{g,h}$ does contain a complete intersection, then $\mathfrak{Th}^r_{g,h}$ has a component
 whose general element is a complete intersection.  
  We work over the complex field $\C$.
  
  This work is part of the Ph.D. Thesis of the second named author and it is partially supported by MIUR (Italy) and by GNSAGA of INdAM.
  
  We thank the referee for several useful remarks.
 
\section{Surjectivity}
 
Let $X$ be a smooth projective variety and $L\in\mathrm{Pic}(X)$. Let $b>a>0$ be two integers. 
Let us fix a very ample line bundle $\mathcal{O}_X(1)$ on $X$ and set $n=h^0(\mathcal{O}_X(1))-1$. Then we use $|\mathcal{O}_X(1)|$ to embed $X$ inside $\P^n$.
Following \cite{W87}, for any $e\in\N_{>0}$ and for any integers $b>a>0$ we have the commutative diagram 
  \begin{equation}\label{diagram1}\begin{diagram}
 H^0(\P^n,\mathcal{O}_{\P^n}(ae))\otimes H^0(\P^n,\mathcal{O}_{\P^n}(be)) & \rTo^{\gamma_{a,b}(\P^n,\mathcal{O}_{\P^n}(e))} & H^0(\P^n,\Omega^1_{\P^n}(ae+be)) \\
  \dTo & & \dTo_{\alpha_{ae+be,X}} \\
  H^0(X,\mathcal{O}_{X}(ae))\otimes H^0(X,\mathcal{O}_{X}(be)) & \rTo^{\gamma_{a,b}(X,\mathcal{O}_{X}(e))} & H^0(X,\Omega^1_{X}(ae+be)).
 \end{diagram}\end{equation}
As in \cite{W87}, we show that the top map is always surjective.
\begin{lemma}\label{lemma1}
 For any $e\in\N_{>0}$ and for any integers $b>a>0$ the map 
 \[
 \gamma_{a,b}(\P^n,\mathcal{O}_{\P^n}(e)): H^0(\P^n,\mathcal{O}_{\P^n}(ae))\otimes H^0(\P^n,\mathcal{O}_{\P^n}(be))\to H^0(\P^n,\Omega^1_{\P^n}(ae+be))
 \]
 is surjective.  
 In particular, the map $\gamma_{a,b}(X,\mathcal{O}_{X}(e))$ is surjective if the map $\alpha_{ae+be,X}$ is surjective.
\end{lemma}
\begin{proof}
 Let us call for the moment $\gamma:=\gamma_{a,b}(\P^n,\mathcal{O}_{\P^n}(e))$.  First we show that the map $\gamma$ is not identically zero.
 Fix homogeneous coordinates $X_0,\ldots,X_n$ for $\P^n$ and consider, for instance, the mononomials 
 $X_0^{ae}$ and $X_1^{be}$. Then
 \[ 
  \gamma(X_0^{ae}\otimes X_1^{be})=abeX_0^{ae}X_1^{be}\left(\frac{\de X_0}{X_0}-\frac{\de X_1}{X_1}\right)
 \]
 which is a non trivial function.
 Notice also that $\gamma$ is a $\mathrm{GL}(n+1)$-equivariant map. For every $t\geq2$ the space $H^0(\P^n,\Omega_{\P^n}^1(t))$ corresponds to
 the kernel of the multiplication map 
 \[
 H^0(\P^n,\mathcal{O}_{\P^n}(1))\otimes H^0(\P^n,\mathcal{O}_{\P^n}(t-1))\to H^0(\P^n,\mathcal{O}_{\P^n}(t)),
 \]
 and hence $\Omega^1_{\P^n}(t)$ is a stable homogeneous bundle. The space $H^0(\P^n,\Omega^1_{\P^n}(t))$ is thus an irreducible $\mathrm{GL}(n+1)$-representation
 for every integer $t\geq 2$. Indeed it corresponds to the Young diagram $(t-1,1)$. This means that the codomain $H^0(\P^n,\Omega^1_{\P^n}(t))$ has
 no proper $GL(n+1)$-invariant subspaces and hence the map $\gamma$ is surjective, since it is equivariant and non-zero.
 \end{proof}

We study the surjectivity of the right-vertical map $\alpha_{t,X}$. For every $t>0$ it factors via two maps:
\begin{equation}\label{eq:rhobeta}
H^0(\P^n,\Omega^1_{\P^n}(t))\stackrel{\rho_{t,X}}{\longrightarrow} H^0(X,\Omega^1_{\P^n|X}(t))\stackrel{\beta_{t,X}}{\longrightarrow}H^0(X,\Omega^1_X(t))
\end{equation}
 where $\rho_{t,X}$ is the restriction map and $\beta_{t,X}$ comes from the cohomology of the conormal sequence.
Considering indeed the conormal sequence $0\to N^*_X(t)\to \Omega^1_{\P^n|X}(t)\to \Omega^1_X(t)\to 0$, we have that the second one is surjective if $h^1(N^*_X(t))=0$.
When $X$ is a complete intersection we have the following results on the first cohomology of the twisted conormal bundle.
\begin{lemma}\label{lemma2}
 Assume that $X$ is a complete intersection of $n-k$ hypersurfaces of degree $d_1\geq \cdots \geq d_{n-k}\geq 2$.
 Then 
 \[
 h^1(N^*_X(t))=0 \iff k\geq 2 \ \text{ or } \  \begin{cases} k=1,\\ t>2d_1+d_2+\cdots+d_{n-1}-n-1.
                      \end{cases}
 \]
\end{lemma}
\begin{proof}
 Recall that $N^*_X(t)=\bigoplus_{i=1}^{n-k}\mathcal{O}_X(t-d_i)$. If $k\geq 2$ then we have that the first cohomology group $H^1(X,\mathcal{O}_X(c))$ vanishes for all $c\in\Z$ since $X$ is a complete intersection.
 If $k=1$ then $\omega_X=\mathcal{O}_X(\sum_i d_i-n-1)$ and hence $N^*_X(t)$ does not have special terms if and only if $t>2d_1+d_2+\cdots+d_{n-1}-n-1$.
\end{proof}

In this situation we can prove the surjectivity of the weighted Gaussian maps in some ranges. In order to do so, we need one more lemma.
\begin{lemma}\label{lemma3}
 Let $X\subset\P^n$ be a smooth variety. Assume $h^1(\P^n,\mathcal{I}_X(t-1))=0$ and that the homogeneous ideal of $X$ contains no minimal generator of degree $t$.
 Then the map $\rho_{t,X}:H^0(\P^n,\Omega^1_{\P^n}(t))\to H^0(X,\Omega^1_{\P^n|X}(t))$ is surjective.
\end{lemma}
 The last hypothesis is equivalent to asking for the surjectivity of the map $$V\otimes H^0(\P^n,\mathcal{I}_{X}(t-1))\longrightarrow H^0(\P^n,\mathcal{I}_{X}(t)),$$ where $V$ is the $(n+1)$-dimensional vector space
 $V:=H^0(\P^n,\mathcal{O}_{\P^n}(1))$.
 \begin{proof}
 From the short exact sequence
 \[
 0\to\mathcal{I}_X\to \mathcal{O}_{\P^n}\to\mathcal{O}_X\to 0
 \]
 tensored by $\Omega^1_{\P^n}(t)$, it follows that the map $\rho_{t,X}$ is surjective if the first cohomology group $H^1( \mathcal{I}_X\otimes \Omega^1_{\P^n}(t))$ vanishes.
 But the latter follows from our assumptions by considering the Euler sequence for $\P^n\simeq\P(V)$ twisted by the ideal sheaf $\mathcal{I}_X(t)$
 \begin{equation}\label{eq:eulseq}
 0\to \mathcal{I}_X\otimes \Omega^1_{\P^n}(t)\to \mathcal{I}_X(t-1)\otimes V \to \mathcal{I}_X(t) \to 0.
 \end{equation}
 and hence the proof is over.
 \end{proof}
 \begin{prop}
 Let $X\subset \P^n$ be a smooth complete intersection of $n-k$ hypersurfaces of degree $d_1\geq \cdots \geq d_{n-k}\geq 2$. 
 Then the weighted Gaussian map $\gamma_{a,b}(X,\mathcal{O}_{X}(1))$ is surjective if $a+b\neq d_i$ and
 \[
  k\geq 2 \ \text{ or } \  \begin{cases} k=1,\\ t>2d_1+d_2+\cdots+d_{n-1}-n-1.
                      \end{cases}
 \]
 \end{prop}
 \begin{proof}
 Defining $t=a+b$ we know by Lemma \ref{lemma1} that $\gamma_{a,b}(X,\mathcal{O}_{X}(1))$ is surjective if the map $\alpha_{t,X}$ is surjective.
 By Lemma \ref{lemma2} we know that the map $\beta_{t,X}$ is surjective if and only if $k\geq 2$ or $k=1$ and $t>2d_1+d_2+\cdots+d_{n-1}-n-1$. 
 Since $X$ is projectively normal one has $h^1(X,\mathcal{I}_X(t-1))=0$. Furthermore, we have just seen in the previous proof that the map $\rho_{t,X}$ is surjective
 if $h^1(\mathcal{I}_X\otimes \Omega^1_{\P^n}(t))=0$. 
 In this situation the ideal sheaf $\mathcal{I}_X$ is a
 quotient of $\bigoplus_i \mathcal{O}_{\P^n}(-d_i)$ because it is minimally generated by the $n-k$ forms of degree $d_1,\ldots,d_{n-k}$ defining $X$. 
 The map $\alpha_{t,X}$ is thus surjective if $h^1(\Omega_{\P^n}^1(t-d_i))=0$ for all $i$, that is $t\neq d_i$ for all $i$.
 \end{proof}

 Let us move to the general situation where $C$ is a smooth curve and $L$ a very ample line bundle on it.
 \begin{teo}\label{teo}
 Let $C$ be a smooth curve and $L\in\mathrm{Pic}(C)$ be a very ample line bundle on $C$. Denote by $X\subset \P^n$ the linearly normal embedding of $C$ induced by $|L|$.
 Let $t_0\in\N$ be an integer such that $t_0\geq 3$. 
 Assume that \begin{enumerate}
              \item $h^1(\P^n,\mathcal{I}_X(t-1))=0$ for every $t\geq t_0$,
	      \item the map $H^0(\P^n,\mathcal{O}_{\P^n}(1))\times H^0(\P^n,\mathcal{I}_{X}(t-1))\to H^0(\P^n,\mathcal{I}_{X}(t))$ is surjective for every $t\geq t_0$,
	      \item $h^1(\mathcal{O}_X(1))=0$. 
             \end{enumerate}
 Then, for every $t\geq t_0$ and for every $a,b>0$ such that $a+b=t$, the weighted Gaussian map $\gamma_{a,b}(X,\mathcal{O}_X(1))$ is surjective.
 \end{teo}
 \begin{proof}
 Notice that hypothesis $(2)$ implies that the homogeneous ideal of $X$ is generated by forms of degree $\leq t_0-1$ and hence  $X$ is scheme-theoretically cut out by forms of degree $t_0-1$.
 By Lemma \ref{lemma1} and the factorization of $\alpha_{t,X}$ given in (\ref{eq:rhobeta}), it is enough to prove that $\rho_{t,X}$ and $\beta_{t,X}$ are surjective for all $t\geq t_0$.
 The first map is surjective by Lemma \ref{lemma3}. 
 Recall that in order to prove the surjectivity of the second map it is enough to show that $h^1(N^*_X(t))=0$ for every $t\geq t_0$. Since $X$ is scheme-theoretically cut out by forms of degree $t_0-1$,
 the ideal $\mathcal{I}_X(t_0-1)$ is spanned by its global section.
 Therefore its quotient
 $N^*_X(t_0-1)=\mathcal{I}_X/\mathcal{I}_X^2(t_0-1)$ is spanned by global sections. Hence
 for some $N\in\N$ we have an exact sequence of $\mathcal{O}_X$-sheaves
 \[
 0\to E \to \mathcal{O}_X(1)^{\oplus N}\to N^*_X(t_0) \to 0.
 \]
 Since $X$ is a curve, we have $h^2(X,E)=0$. Furthermore, from $h^1(\mathcal{O}_X(1))=0$ we obtain $h^1(N^*_X(t_0))=0$ and the same argument works for $t\geq t_0$.
 \end{proof}
 From this we can deduce some special cases as the following ones.

 \begin{cor}\label{cor1}
 Let $C$ be a smooth curve of genus $g$ and $L\in\mathrm{Pic}(C)$ a line bundle of degree $\deg(L)\geq 2g+2$. Denote by $X\subset \P^n$ the linearly normal embedding of $C$ induced by $|L|$.
 Then the map $\gamma_{a,b}(X,\mathcal{O}_X(1))$ is surjective for all positive integers $a,b$ such that $a+b\geq 3$.
 \end{cor}
 \begin{proof}
 Since $\deg(L)\geq 2g+2$, one has $h^1(L)=0$ and $L$ is very ample. In particular, $X$ is projectively normal (\cite[p. 140]{ACGH85}) and
 one has $h^1(\mathcal {I}_X(2)) =0$.
Since $\deg(L)\geq 2g+2$, the homogeneous ideal of $X$ is generated by quadrics (\cite{G84} and \cite[p. 302]{GL88}).
Hence, for $t\geq 3$, the map $H^0(\P^n,\mathcal{O}_{\P^n}(1))\times H^0(\P^n,\mathcal{I}_{X}(t-1))\to H^0(\P^n,\mathcal{I}_{X}(t))$ is surjective.
 We can then apply Theorem \ref{teo} with $t_0=3$ and this completes the proof.
 \end{proof}

 \begin{cor}\label{cor2}
 Let $C$ be a smooth curve of genus $g\geq 3$. Assume that $C$ is neither hyperelliptic nor trigonal nor isomorphic to a plane quintic and let $X\subset \P^{g-1}$ be its canonical model. Then
 the map $\gamma_{a,b}(X,\omega_X)$ is surjective for all $a,b>0$ such that $a+b\geq 4$.
 \end{cor}
 \begin{proof}
 Every canonically embedded smooth curve is projectively normal by a theorem of Max Noether \cite[p.117]{ACGH85} and hence $h^1(\mathcal {I}_X(2)) =0$. 
 Since $C$ is neither trigonal nor a plane quintic, a theorem of K. Petri says that
 the homogeneous ideal of $X$ is generated by quadrics \cite[p. 131]{ACGH85}. 
 Hence condition $(2)$ of Theorem \ref{teo} is satisfied for $t_0=3$. The same argument used in the proof of
 Theorem \ref{teo} shows that
 $N^*_X(a+b-2)$ is spanned by global section and we have thus an analogous exact sequence
  \[
 0\to E \to \mathcal{O}_X(2)^{\oplus N}\to N^*_X(a+b) \to 0.
 \]
 Since
 $h^1(\mathcal{O}_X(2))=h^1(\omega_X^{\otimes 2})=0$ we get $h^1(N^*_X(a+b))=0$ and hence we conclude that the map $\beta_{a+b,X}$ is surjective.
 Since the homogeneous ideal of $X$ is generated by forms of degree $\leq a+b-1$ (actually, of degree $\leq a+b-2$), from Lemma \ref{lemma3} we deduce that the map $\rho_{a+b,X}$ is surjective.
Hence $\gamma_{a,b}(X,\omega_X)$ is surjective.
 \end{proof}

\section{Gaussian maps and $h$-theta-characteristics}

 In this section we identify some objects in the kernel of the Gaussian maps, in particular the ones of the kind $v^{\otimes a}\otimes v^{\otimes b}$ for $v\in H^0(L)$. We
 get a lower bound on the dimension of the kernel, which in general is far from being sharp, but that it is achieved in the case of interest for us.
 Consider the multiplication map
 \[
 \eta_{a+b}: \mathrm{Sym}^{a+b}(H^0(L))\to H^0(L^{\otimes a+b}).
 \]
 \begin{lemma}\label{teo3}
  For every $a,b>0$, the kernel of $\gamma_{a,b}(X,L)$ has dimension $\geq\rank(\eta_{a+b})$.
 \end{lemma}
 \begin{proof}
 Let $v$ be an element in $H^0(L)$. Define $\sigma:=v^{\otimes a}\in H^0(L^{\otimes a})$ and $\tau:=\sigma^{\otimes b}\in H^0(L^{\otimes b})$. We have
 \[
 \gamma_{a,b}(X,L)(\sigma\otimes\tau)=\frac{\tau^{a+1}}{\sigma^{b-1}} d\left( \frac{\sigma^{b}}{\tau^a} \right)=0
 \]
 where the first equality follows from the proof of \cite[Lemma 1]{BF06} and the second follows since $\sigma^{b}=\tau^a$ by definition. This means that the kernel of $\gamma_{a,b}(X,L)$ contains the linear span of all elements of the kind $ v^a\otimes v^b$ for $v \in H^0(L)$.
 In order to prove $\dim(\ker\gamma_{a,b}(X,L))\geq \rank(\eta_{a+b})$ it is enough to show that, given the multiplication map
 $
 \nu_{a,b}:H^0(L^{\otimes a})\otimes H^0(L^{\otimes b})\to H^0(L^{\otimes a}\otimes L^{\otimes b}),
 $
one has
 \[
 \dim(\nu_{a,b}(\ker\gamma_{a,b}(X,L)))\geq \rank(\eta_{a+b}).
 \]
 We have just shown that $\ker\gamma_{a,b}(X,L)$ contains all elements of the kind $v^{\otimes a+b}$ for $v\in H^0(L)$.
 Since we work over algebraically closed base field with characteristic zero, 
 any symmetric polynomial of degree $a+b$ is a linear combination of $(a+b)$-th powers of linear forms, that is
 $$\mathrm{Im}(\eta_{a+b})\subset\mathrm{lin\_span}\{v^{\otimes a+b}\in H^0(L^{\otimes a+b}) \ \mid \ v\in H^0(L)\}\subset \nu_{a,b}(\ker\gamma_{a,b}(X,L)).$$
 The thesis follows.
 \end{proof}
 
 Notice that, when $a=1$, $b=h-1$ and $h$ is even, we are in the situation of Theorem 3 of \cite{Fon02} and the map $\gamma_{a,b}(C,L)$ corresponds to the map
 \[\begin{array}{cccc}
 \mu_h:=\gamma_{1,h-1}(C,L): & H^0(C,L)\otimes H^0(C,K_C-L) & \to & H^0(C,2K_C)\\
                              & \sigma \otimes \tau & \mapsto & (h-1)\tau\de\sigma-\sigma\de\tau
 \end{array} \]
  Notice the misprint in the statement of \cite[Theorem 3]{Fon02} and \cite[Theorem 1]{BF06}.
 Since in Theorem 3 of \cite{Fon02} only the case $h$ even is considered, we restate here the theorem in its full generality.
 \begin{teo}\label{teofon}
  Let $\mathfrak{Th}^r_{g,h}$ be the space defined in (\ref{deflocus}) for $g,h\geq 2$ and define the map
  \[\begin{array}{cccc}
   \mu_h:= & H^0(C,L)\otimes H^0(C,K_C-L) & \to & H^0(C,2K_C)\\
                              & \sigma \otimes \tau & \mapsto & (h-1)\tau\de\sigma-\sigma\de\tau.
 \end{array}
  \]
  Then $T_C\left(\mathfrak{Th}^r_{g,h}\right)=\left(\mathrm{coker}\mu_h\right)^*$.
 \end{teo}
 \begin{proof}
 The proof of \cite[Theorem 3]{Fon02} works verbatim when $h\ge 3$ is an odd integer. Indeed, just use $\frac{1}{h}$ instead of $\frac{1}{2m}$ in 
 the proof of \cite[Lemma 3]{Fon02} and use $h$ instead of $2m$ in \cite[\S 3.3]{Fon02}.
 \end{proof}
 
 In particular, when $C$ is a smooth and connected projective curve, Proposition 1 of \cite{BF06} gives us the bound
 \begin{equation}\label{boundleq}
 \rank(\gamma_{a,b}(C,L))\geq h^0(L^{\otimes a})+h^0(L^{\otimes b})-3,
 \end{equation}
 where the inequality is strict if $bs-at\neq 0$, with $s=h^0(L^{\otimes a})-1$ and $t=h^0(L^{\otimes b})-1$.
 Recalling the definition (\ref{deflocus}) of $\mathfrak{Th}^r_{g,h}$ given in the Introduction, 
 Lemma \ref{teo3} provides a bound on the dimension of the Zariski tangent space $\mbox{Zar}_{[C]}\mathfrak{Th}^r_{g,h}$ of $ \mathfrak{Th}^r_{g,h}$ at $[C] \in \mathfrak{Th}^r_{g,h}$ inside
 the Zariski tangent space $H^0(C,K_C^{\otimes 2})^\vee$ of the deformation space of $C$. In particular, one has
 \begin{equation}\label{eq:conj}
 \rank(\mu_h)=\mathrm{codim}\mbox{Zar}_{(C,L)}\mathfrak{Th}^r_{g,h} \leq (r+1)\left(r+1 +(g-1)\frac{h-2}{h}\right)-\rank(\eta_{h}).
 \end{equation}
 In his work G. Farkas studied the classical locus $\mathcal{S}^r_g:=\{[C,L]\in\mathcal{S}_g \mid h^0(L)\geq r+1,h^0(L)\equiv_2 r+1\}$ cor\-responding to the choice $h=2$ and
 he showed that, for small values of $r$, there is a specific integer $g(r)$ such that for every $g\geq g(r)$ there exists a component of $\mathcal{S}^r_g$ realizing the bound.
 He also made a conjecture
 (see \cite[Conjecture 3.4]{Far05}), recently proved by \cite{Be13},
 on the existence of such a component for every $r\geq 3$ and for every $g\geq {r+2 \choose 2}$.
 
 When $h$ is general and $r=0$, we know the exact dimension of the tangent space of  $\mathfrak{Th}^0_{g,h}$ at any $[C]$ with $h^0(L)=1$, by seeing it as a subspace of the
 tangent space $H^0(K_C^{\otimes 2})^\vee$ of the local deformation space of $C$.
 \begin{teo}\label{teo:sect3}
  For every $g,h\geq 2$ and every $[C] \in \mathfrak{Th}^0_{g,h}$ with an $h$-theta $L$ satisfying $h^0(L)=1$, the Zariski tangent space at $[C]$
  has codimension $(g-1)(h-2)/h$ in the tangent space $H^0(C,K _C^{\otimes 2})^\vee $ of the local deformation space of $C$.
 \end{teo}
 \begin{proof}
 Let $C$ be a smooth curve and $L\in \mathrm{Pic}(X)$ an $h$-th canonical root such that $h^0(L)=1$. Notice that, when $r=0$, the rank of the map $\eta_s$ is always equal to $1$ for every $s\in \N$.
 From (\ref{boundleq}) we know that
\begin{equation}\label{eqo+}
 \rank(\mu_h)\geq h^0(L)+h^0(K_C-L)-3=h^0(K_C-L)-2
 \end{equation}
 and from Lemma \ref{teo3} we have
 \begin{equation}\label{eqteo}
 \rank(\mu_h)\leq h^0(L)\cdot h^0(K_C-L)-1=h^0(K_C-L)-1
 \end{equation}
 When $g\geq 2$ and $h >2$, the inequality in (\ref{eqo+}) is strict because $s=h^0(L)-1=0$ and $t=h^0(K_C-L)-1=g-\frac{2g-2}{h}\neq 0$ 
 (\cite[Proposition 1]{BF06}). Furthermore, if $h=2$ then
 $h^0(K_C-L)=1$ and thus, from (\ref{eqteo}), we have $\rank(\mu_2)=h^0(K_C-L)-1=0$ even in this case. We can thus conclude that, for all $g,h\geq 2$, the map $\mu _h$
 has always rank $h^0(K_C-L)-1=(g-1)(h-2)/h$.
 \end{proof}

 For any $g,r$ and $h$ we cannot expect a complete intersection $C$ to be an isolated point of $\mathfrak{Th}^r_{g,h}$, since curves of this kind move in families. However, we show that the dimension of a component of $\mathfrak{Th}^r_{g,h}$ 
 containing a complete intersection is the smallest possible one. In other words, we show that if a component of 
 $\mathfrak{Th}^r_{g,h}$ contains a complete intersection curve, then its general element is again a complete intersection.
 Let $C\subset \P^n$ be a smooth complete intersection of hypersurfaces of degree $d_1,\ldots,d_{n-1}$ with
 \[
 d_1\geq \cdots \geq d_{n-1}\geq 2, \qquad n\geq 2.
 \]
 If $n=3$ assume $d_1\geq 3$ and if $n=2$ assume $d_1\geq 4$. Let us call $\xi:=\sum_i d_i-n-1$ and recall that $\omega_C=\mathcal{O}_C(\xi)$. We need a preliminary lemma.

 \begin{lemma}\label{lemma3.1}
  The cokernel of the cotangent map $$\alpha: H^0(C,\Omega_{\P^n|C}^1(\xi))\to H^0(C,\omega^{\otimes 2}_C)$$ has dimension
  $h^0(N_C)-(n+1)^2+1$.
 \end{lemma}
 \begin{proof}
 We use the conormal exact sequence
 \begin{equation}\label{eq:conseq}
  0\to N^*_C\left(\xi\right)\to \Omega_{\P^n|C}^1\left(\xi\right)\to \omega_C\left(\xi\right)\to 0.
 \end{equation}
 From its cohomology and Serre duality we find
 \[
 \dim\mathrm{coker}(\alpha)=h^0(N_C)-h^1( \Omega^1_{\P^n|C}(\xi) ).
 \]
 In order to compute $h^1(\Omega^1_{\P^n|C}(\xi))$, define the vector space $V:=H^0(\P^n,\mathcal{O}_{\P^n}(1))=H^0(C,\mathcal{O}_C(1))\cong \C^{n+1}$ and 
 consider the restriction to $C$ of the Euler sequence of $\Omega_{\P^n}^1$:
 \begin{equation}\label{eeulseq}
  0\to \Omega_{\P^n|C}\left(\xi\right)\to V\otimes \mathcal{O}_C\left(\xi-1\right)\stackrel{\mu}{\to} \mathcal{O}_C\left(\xi\right)\to 0.
 \end{equation}
 Since $\xi-1\geq 0$ and $C$ is projectively normal, the map $\mu$ is a surjective multiplication map on global sections. From the long exact
 sequence of cohomology of (\ref{eeulseq}) and Serre duality, it thus follows that $h^1(\Omega^1_{\P^n|C}(\xi))=(n+1)^2-1$, proving the lemma.
\end{proof}

Now define $g:=\frac{\xi\prod_i d_i}{2}+1$, $\zeta:=\frac{\xi}{h}$, $r:=h^0(C,\mathcal{O}_C(\zeta))-1$
and consider the diagram
\begin{equation}\begin{diagram}\label{diagram-us}
  H^0(\P^n,\mathcal{O}_{\P^n}(\zeta))\otimes H^0(\P^n,\mathcal{O}_{\P^n}((h-1)\zeta))  & \rTo &  H^0(\P^n,\Omega^1_{\P^n}(\xi))\\
  \dTo & & \dTo_\alpha\\
  H^0(C,\mathcal{O}_C(\zeta))\otimes H^0(C,\mathcal{O}_C((h-1)\zeta))  & \rTo^{\gamma_{1,h-1}} &  H^0(C,\Omega^1_C(\xi)).\\
 \end{diagram}\end{equation}
 
 We can state the following Theorem.
\begin{teo}\label{teofinal}
For every $h\geq 2$ and $d_i, h, g, r$ as above, the curve $[C]$ is a smooth point of an irreducible component of $\mathfrak{Th}^r_{g,h}$ whose general member is a complete intersection curve.
\end{teo}
\begin{proof}
 Let $\mathbb {D}$ be the subset of the Hilbert scheme $\mbox{Hilb}(\mathbb {P}^n)$ of $\mathbb {P}^n$ parametrizing the smooth complete intersection of type $d_1,\dots , d_{n-1}$ of $\mathbb {P}^n$. The set $\mathbb {D}$ is an open subset of $\mbox{Hilb}(\mathbb {P}^n)$ (\cite[p. 236]{S06}). Hence for
 each $C\in \mathbb {D}$ the tangent space of $\mathbb {D}$ at $C$ has dimension $h^0(N_C)$. Since $\mathbb {D}$ is smooth (\cite[p. 236]{S06}),
 we have $\dim (\mathbb {D}) = h^0(N_C)$. The set $\mathbb {D}$ is irreducible, because it is parametrized by an open subset of the vector space $\prod _{i=1}^{n-1} H^0(\mathbb {P}^n,\mathcal {O}_{\mathbb {P}^n}(d_i))$. Since $C$ has genus $\ge 2$, only finitely many automorphisms of $\mathbb {P}^n$ sends $C$ into itself. The adjunction formula
 gives $\omega _C \cong \mathcal {O}_C(d_1+\cdots +d_{n-1} -n-1)$. Set $e:= d_1+\cdots +d_{n-1} -n-1$. Since $C$ has only finitely many line bundles $L$ such
 that $L^{\otimes e} \cong \omega _C$, up to projective isomorphism $C$ has only finitely many embeddings into $\mathbb {P}^n$ as a complete intersection of type $d_1,\dots ,d_{n-1}$. Hence the locus $\Gamma\subset \mathcal{M}_g$ of all complete intersection curves is an irreducible variety of dimension 
 $ h^0(N_C)-\dim\mathrm{Aut}(\P^n)=h^0(N_C)-(n+1)^2+1$ and it is contained in an irreducible component $\mathcal{Z}$ of $\mathfrak{Th}^r_{g,h}$.
 From diagram (\ref{diagram-us}), Lemma \ref{lemma1} and the fact that $C$ is projectively normal, we see that the map $\gamma_{1,h-1}(C,\mathcal{O}_C(\zeta))$ has cokernel equal to the cokernel of the map $\alpha$ and hence of dimension exactly $h^0(N_C)-(n+1)^2+1$.
 Then, Theorem \ref{teofon} implies $\dim(T_C\mathcal{Z})=h^0(N_C)-(n+1)^2+1$ and we can conclude that the curve $[C]$ is a
 smooth point of the component. In particular, $\Gamma=\mathcal{Z}$ and hence the claim follows.
\end{proof}

\vspace{0.3cm}
\begin{footnotesize}\noindent\textsc{Universit\`{a} degli Studi di Trento, Dipartimento di Matematica, \\ Via Sommarive 14,  I-38123 Povo (TN)}\end{footnotesize}

\noindent E-mail addresses: \verb=edoardo.ballico@unitn.it=, \verb=pernigotti@science.unitn.it=

\end{document}